\documentclass[11pt]{article}
\usepackage[top=2.5cm, bottom=2.5cm, left=2.5cm, right=2.5cm]{geometry}
\usepackage{graphicx}
\usepackage{makeidx}
\usepackage{graphics}
\graphicspath{ {images/} }
\usepackage{amsmath}
\usepackage{amsthm}
\usepackage{amsfonts}
\usepackage{amssymb}
\usepackage{epsfig}
\usepackage{tabularx}
\usepackage{latexsym}
\usepackage{mathtools}
\usepackage{epstopdf}
\usepackage{array}
\usepackage[utf8]{inputenc}

\newtheorem{lemma}{Lemma}
\newtheorem{conjecture}{Conjecture}

\newtheorem{theorem}{Theorem}
\newtheorem{corollary}{Corollary}
\newtheorem{observation}{Observation}
\newcommand{\floor}[1]{\lfloor #1 \rfloor}
\newcommand{\ceil}[1]{\lceil #1 \rceil}

\newcommand{\iitgaddress}{IIT Guwahati, Assam 781039, India.}
\newcommand{\nusaddress}{National University of Singapore, Singapore.}
\newcommand{\iitgemail}[1]{Email: \texttt{#1@iitg.ernet.in}}
\newcommand{\nusmail}[1]{Email:\texttt{#1@u.nus.edu}}

\title{On the Rectilinear Crossing Number of Complete Uniform Hypergraphs}
\author{
Anurag Anshu\footnote{\nusaddress\ \nusmail{a0109169}} 
\and
Rahul Gangopadhyay\footnote{\iitgaddress\ \iitgemail{r.gangopadhyay}} 
\and
Saswata Shannigrahi\footnote{\iitgaddress\ \iitgemail{saswata.sh}}
\footnote{Corresponding Author}
\and 
Satyanarayana Vusirikala\footnote{\iitgaddress\ 
\iitgemail{vusirikala}}
}

\begin{document}
\maketitle
\begin{abstract}
 
In this paper, we consider a generalized version of the rectilinear crossing
number problem of drawing complete graphs on a plane. The minimum number of
crossing pairs of hyperedges in the $d$-dimensional rectilinear drawing of a
$d$-uniform hypergraph is known as the $d$-dimensional rectilinear crossing
number of the hypergraph. The currently best-known lower bound on the
$d$-dimensional rectilinear crossing number of a complete $d$-uniform hypergraph
with $n$ vertices in general position in $\mathbb{R}^d$ is
$\Omega(\frac{2^d}{\sqrt{d}} \log d) {n \choose 2d}$. In this paper, we improve
this lower bound to $\Omega(2^d) {n \choose 2d}$. We also consider the special
case when all the vertices of a $d$-uniform hypergraph are placed on the 
$d$-dimensional moment curve.
For such complete $d$-uniform hypergraphs with $n$ vertices, 
we show that the number of crossing pairs of hyperedges is 
$\Theta(\frac{4^d}{\sqrt{d}}) {n \choose 2d}$.

 \vspace{.3cm}
\end{abstract}

\section{Introduction}
\label{introduction}
Graph drawing in a plane has been a well-studied area of research for many 
years with 
applications in various fields of computer science, e.g., CAD, database design 
and circuit schematics \cite{TT, TT1}. One particularly interesting drawing of 
a 
graph is the rectilinear drawing, 
defined as an embedding of the graph 
on a plane $({\Re}^2)$ with vertices placed in general position (i.e., no three 
vertices are collinear) 
and edges connecting corresponding vertices as straight line segments.
The rectilinear crossing number of a graph is defined as the minimum number of 
crossing pairs of edges over all rectilinear drawings of the graph \cite{SA}. 
Let us denote the rectilinear crossing number of a complete graph with 
$n$ vertices by ${\overline {cr}_2}(K_n)$. The currently best-known lower and 
upper bounds on ${\overline {cr}_2}(K_n)$ are $0.37997\dbinom{n}{4}+\Theta(n^3)$ 
and $0.380473\dbinom{n}{4}+\Theta(n^3)$, respectively \cite{Alb,FML}. \\
The concept of a hypergraph is a natural extension to the notion of a graph. A 
hypergraph 
$H$ is defined as a pair $H=(V,E)$, where $V$ is the set of vertices and $E$ 
is a set of non-empty subsets of $V$ called hyperedges. A hypergraph in which 
each hyperedge contains exactly $d$ vertices is called a $d$-uniform 
hypergraph. 
Throughout this paper, we consider $d \geq 2$ because the hyperedge set of a 
$1$-uniform hypergraph is only a collection of sets containing one vertex each. 
Let us denote the complete $d$-uniform hypergraph with $n$ vertices by $K^d_n$. 
A \textit{$d$-dimensional rectilinear drawing} of a $d$-uniform hypergraph with 
$n\geq2d$ vertices is 
defined as an embedding of the hypergraph in ${\Re}^d$ with vertices placed in 
general position (i.e., no $d+1$ of the vertices lie on a common 
hyperplane) and hyperedges drawn as $(d-1)$-simplices spanned by the $d$ 
vertices in the corresponding hyperedges \cite{SA}. In a $d$-dimensional 
rectilinear 
drawing of a hypergraph, two hyperedges are said to be \textit{intersecting} if 
they contain a common point in their relative interiors \cite{DP}. 
Two intersecting hyperedges are said to be \textit{crossing} if they are vertex 
disjoint, as shown in Figure \ref{fig:crossing}. This definition is extended to 
define the crossing between a $u$-simplex and a $v$-simplex such that $ 0 \leq 
u\leq d-1$, $0 \leq v \leq d-1$ and all the vertices ($0$-faces) belonging to 
both these simplices 
are in general position in $\mathbb{R}^d$. Note that a $u$-simplex (similarly, 
$v$-simplex) in $\mathbb{R}^d$ is defined as the convex hull $Conv(A)$ of a set 
$A$ of 
$u+1$ points ($v+1$ points) in general position in $\mathbb{R}^d$. Such a 
$u$-simplex and a $v$-simplex are said to be crossing if 
they are vertex 
disjoint and contain a common point in their relative interiors. The 
\textit{$d$-dimensional rectilinear crossing number} of a $d$-uniform 
hypergraph is defined as the minimum number of crossing pairs of hyperedges 
over 
all d-dimensional rectilinear drawings of the hypergraph. Let $\overline 
{cr}_d(H)$ denote the $d$-dimensional rectilinear crossing number 
of a $d$-uniform hypergraph $H$.  In this paper, we use $c_d$ to 
denote ${\overline {cr}_d}({K^d_{2d}})$. It follows that 
${\overline{cr}_d}({K^d_n}) \geq c_d\dbinom{n}{2d}$, as  
the set of $d$-dimensional rectilinear crossings created 
by the $\dbinom{2d}{d}$ hyperedges formed by a particular set of $2d$ vertices 
is disjoint from the set of $d$-dimensional 
rectilinear crossings created 
by the $\dbinom{2d}{d}$ hyperedges formed by another set of $2d$ vertices. 
The best-known lower bound on $c_d$ is $\Omega(\dfrac{{2^d}\log 
{d}}{\sqrt{d}})$ \cite{SA}. Currently, the only known upper bound on $c_d$ is 
the trivial bound  
$c_d \leq \dbinom{2d}{d}=\Theta(\dfrac{4^d}{\sqrt{d}})$. This significant 
gap 
between the currently best-known lower and upper bounds on $c_d$ shows 
that at least one of these bounds can be improved. In this 
paper, we work towards improving the lower bound on $c_d$. \par

\begin{figure}[h]
    \centering
    \includegraphics[width = 0.75\textwidth]{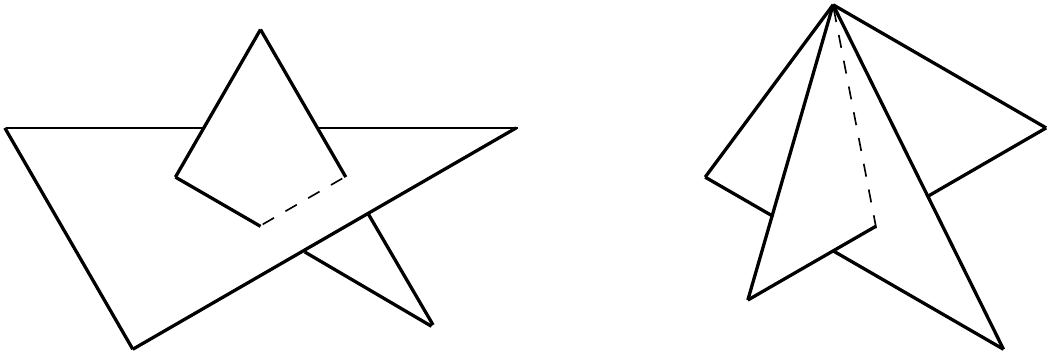}
    \caption{(left) crossing simplices in $\mathbb{R}^3$, (right) intersecting 
simplices in $\mathbb{R}^3$}
    \label{fig:crossing}
\end{figure}

Similar to the rectilinear drawing, the convex drawing of a graph is 
also an 
active area of research \cite{SCH}. The convex drawing of a graph $G$ 
with $n$ vertices is 
defined as an embedding of the graph $G$ in $\mathbb{R}^2$ with all its $n$ 
vertices 
placed as the vertices of a convex $n$-gon (note that it also ensures that the 
vertices are in general position) and edges connecting the 
corresponding vertices drawn as straight line segments \cite{SSZ}. Let 
$cr^*(G)$ denote the convex crossing number of $G$, defined as the minimum 
number of 
crossing pairs of edges over all convex drawings of $G$ \cite{SZ2}. A \textit{$d$-dimensional convex 
drawing} 
of a $d$-uniform hypergraph with 
$n\geq2d$ vertices is 
defined as an embedding of the hypergraph in ${\Re}^d$ with vertices 
placed in convex as well as general
position and hyperedges 
drawn as $(d-1)$-simplices spanned by the $d$ 
vertices in the corresponding hyperedges. The \textit{$d$-dimensional convex 
crossing number} of a $d$-uniform hypergraph is defined as the minimum number 
of 
crossing pairs of hyperedges over all d-dimensional convex drawings of the 
hypergraph.   
Let $cr^*_d(H)$ denote the $d$-dimensional convex crossing 
number 
of a $d$-uniform hypergraph $H$. In this paper, we use $c^*_d$ to denote 
${cr^*_d}({K^d_{2d}})$.
Since it is easy to see that ${cr^*_d}(K^d_n) \geq c^*_d\dbinom{n}{2d}$, 
finding the value of $c^*_d$ is 
also an interesting problem. The above-mentioned result of Anshu and 
Shannigrahi \cite{SA} implies that the lower bound on 
$c^*_d$ is 
$\Omega(\dfrac{{2^d}\log 
{d}}{\sqrt{d}})$. On the other hand,
a trivial upper bound on $c^*_d$ 
is $O(\dfrac{4^d}{\sqrt{d}})$. We are not aware of 
better lower and upper bounds on $c^*_d$ in the literature. We address a 
special case of the problem in this paper, which 
is when all the vertices of 
$K^d_{2d}$ 
lie 
on the \textit{$d$-dimensional moment curve} 
$\gamma = \{(t,t^2,t^3,\ldots,t^d): t \in\mathbb{R}\}$. Our motivation to work 
with 
this special case is the Upper Bound 
Theorem~\cite{MM} which states that
the \textit{$d$-dimensional cyclic polytope} (i.e., the polytope 
whose vertices are all placed on the $d$-dimensional moment curve) has  
the maximum number of faces among all $d$-dimensional convex polytopes having $n$ 
vertices. In this paper, we explore if the placement of $n$ vertices on the $d$-dimensional moment curve
also maximizes the number of crossing pairs of hyperedges in a 
$d$-dimensional convex drawing of 
$K^d_{n}$. We use $c^m_d$ to denote the number of crossing pairs of hyperedges 
of 
$K^d_{2d}$, when all the vertices of 
$K^d_{2d}$ are placed as the vertices of a $d$-dimensional cyclic polytope.

\subsection{Previous Works}

As mentioned earlier, the best-known lower bound on $c_d$ is 
$\Omega(\dfrac{{2^d}\log{d}}{\sqrt{d}})$.
 The proof uses Gale Transform to reduce 
the crossing number problem to a linear separation problem. The Gale Transform is a technique to transform a sequence of 
$m$ 
points $P=\big<p_1, p_2, \ldots, p_m\big>$ in $\mathbb{R}^d$ to a sequence of 
$m$ 
vectors 
$D(P)=\big<v_1, v_2, \ldots, v_m\big>$ in $\mathbb{R}^{m-d-1}$ for $m \geq d+1$. 
(In 
Section 
\ref{tused}, we 
discuss the Gale Transform and its properties in detail.) 
A \textit{linear separation} of a vector sequence $D(P)$ is the partitioning of 
$D(P)$ by 
a 
hyperplane, passing through the origin, into two sets. A linear 
separation of 
$D(P)$ is called \textit{proper} if one of the sets contains $ \lceil 
\small{\dfrac{m}{2}}\rceil$ vectors and the other contains $ 
\lfloor\small{\dfrac{m}{2}}\rfloor$ vectors. For a given set of $d+4$ points in 
$\mathbb{R}^d$, the Gale 
Transform ensures that there exists a bijection between the crossing pairs 
of $\lfloor\small{\dfrac{d+2}{2}}\rfloor$ and 
$\lceil\small{\dfrac{d+2}{2}}\rceil$-simplices in $\mathbb{R}^d$ and the proper 
linear 
separations of $d+4$ vectors in $\mathbb{R}^3$ \cite{JM}.
In order to calculate the lower bound on $c_d$, Anshu and Shannigrahi \cite{SA} 
chose a set of $d+4$ vertices from the set of $2d$ vertices of $K^d_{2d}$ 
in $\mathbb{R}^d$. The Gale Transform of these $d+4$ vertices gives $d+4$ 
vectors in general position in $\mathbb{R}^3$ (i.e., any subset containing $3$ 
vectors 
spans 
$\mathbb{R}^3$). To get a proper linear separation of this set of $d+4$ vectors 
in 
$\mathbb{R}^3$, they used the Ham-Sandwich Theorem stated below.

\vspace{.1cm}

\noindent\textbf {Ham-Sandwich Theorem.}~\cite{JM}
\textit{There exists a $(d-1)$-hyperplane $h$ which simultaneously bisects $d$ 
finite point sets $P_1, P_2, \ldots, P_d$ in $\mathbb{R}^d$, such that each of 
the 
open half-spaces created by $h$ contains at most 
$\lfloor\small{\dfrac{|P_i|}{2}}\rfloor$ 
points for each of the sets $P_i, 1 \leq i \leq d$.}

\vspace{.1cm}

Using this Theorem, Anshu and Shannigrahi \cite{SA} proved the existence of 
$\Theta(\log d)$ distinct proper 
linear separations of the set of $d+4$ vectors mentioned above. 
As discussed earlier, each proper linear separation of $d+4$ 
vectors in $\mathbb{R}^3$ corresponds to a crossing between $ 
\lfloor\small{\dfrac{d+2}{2}}\rfloor$ and $ 
\lceil\small{\dfrac{d+2}{2}}\rceil$ simplices in $\mathbb{R}^d$. They extended 
the 
crossings between the lower dimensional simplices (crossings between $ 
\lfloor\small{\dfrac{d+2}{2}}\rfloor$ and $ 
\lceil\small{\dfrac{d+2}{2}}\rceil$-simplices) to the 
crossings between $(d-1)$-simplices to get the bound 
$c_d=\Omega(\dfrac{{2^d}\log 
{d}}{\sqrt{d}})$. In particular, they showed that 
$c_4 \geq 4$. They also constructed an arrangement of $8$ vertices of a complete
$4$-uniform hypergraph in $\mathbb{R}^4$ having $4$ crossing pairs of 
simplices. This arrangement established that $c_4 = 4$. This is the first 
non-trivial result on $c_d$ for small values of $d$, as it is easy to 
observe that $c_2 = 0 $ and $c_3 = 1$. \par

\subsection{Our Contribution}
\label{ourc}
    
In this paper, we first improve the lower bound on 
$c_d$  in Section~\ref{gencase}.
\begin{theorem}
 \label{thm11}
 The $d$-dimensional rectilinear crossing number of a complete $d$-uniform 
hypergraph with $2d$ vertices in 
general position in $\mathbb{R}^d$ is $\Omega (2^d)$.
\end{theorem}
\begin{corollary}
 ${\overline {cr}_d}({K^d_n}) = \Omega (2^d)\dbinom{n}{2d}$.
\end{corollary}

\noindent We derive the exact value of $c_d^m$ in
Section~\ref{cncp}. For a sufficiently large $d$, it implies that there exists a constant $c > 0$ such that the number of crossing pairs of 
hyperedges when all $n$ vertices of $K^d_{n}$ are placed on the $d$-dimensional moment curve is at least $c$ times the number of crossing pairs of
hyperedges in any $d$-dimensional convex drawing of $K^d_{n}$.

\begin{theorem}
\label{momntup1}
\[c_d^m= \begin{cases} 
      
\dbinom{2d-1}{d-1}-\sum\limits_{i=1}^{\frac{d}{2}}\dbinom{d}{i}\dbinom{
d-1 } { i-1} & if~d~is~even  \\
      
\dbinom{2d-1}{d-1} 
 
-1-\sum\limits_{i=1}^{\floor{\frac{d}{2}}}\dbinom{d-1}{i}\dbinom{
d}{i} & if~d~is~odd\\ 
   \end{cases}\\
\]
   
$~~~~~~~~~~~~~~~~~~~~~~~~~= \Theta\Bigg(\dfrac{4^d}{\sqrt{d}}\Bigg)$

\end{theorem}
\begin{corollary}
\label{momentcurvemaximizes}
The number of crossing pairs of hyperedges when all the $n$ vertices of 
$K_n^d$ are placed on the $d$-dimensional moment curve is 
$\Theta\Bigg(\dfrac{4^d}{\sqrt{d}}\Bigg)\dbinom{n}{2d}$. 
\end{corollary}

\noindent In Section \ref{conjecture3d}, we prove the following Theorem which 
implies that the $3$-dimensional convex crossing number of $K_6^3$ is $3$. It also
shows that the placement of $n$ vertices on the $3$-dimensional moment curve maximizes the number of crossing pairs of hyperedges in a 
$3$-dimensional convex drawing of $K^3_{n}$. For $d > 3$, we do not know if the placement of $n$ vertices on the $d$-dimensional moment curve maximizes the number of crossing pairs of hyperedges in a $d$-dimensional convex drawing of $K^d_{n}$.
\begin{theorem}
\label{convex3d}
The number of crossing pairs of hyperedges of $K_6^3$ is $3$ when all the 
vertices of $K_6^3$ are in convex as well as general position in $\mathbb{R}^3$.
\end{theorem}

\begin{corollary}
${cr^*_3}({K^3_{n}}) = 3 \dbinom{n}{6}$.
\end{corollary}

\section{Gale Transform}
\label{tused}

In this Section, we discuss the Gale Transform and its properties in 
detail. As mentioned earlier, the Gale Transform 
transforms a sequence of  $m \geq d+1$ points $P=\big<p_1, p_2, \ldots, 
p_m\big>$ in 
$\mathbb{R}^d$ (such that the affine hull of the points $p_1, p_2, \ldots, p_m$ 
is 
$\mathbb{R}^d$) to 
a sequence of $m$ vectors 
$D(P)=\big<v_1, v_2, \ldots, v_m\big>$ in $\mathbb{R}^{m-d-1}$. When $m \leq 
2d$, this 
technique 
helps in analyzing the properties of the point sequence $P$  by analyzing the 
properties of $D(P)$ in a lower dimensional space. Let $p_i$ 
denote the $i^{th}$ point in the point sequence $P$ with 
coordinates $(x_1^i, x_2^i, \ldots, x_d^i)$. To obtain the Gale 
Transform of $P$, let us consider the 
following matrix $M(P)$.
\begin{center}
$
M(P) = 
\begin{bmatrix}
x_1^1&x_1^2&\cdots & x_1^m  \\
x_2^1&x_2^2&\cdots & x_2^m \\
\vdots & \vdots & \vdots & \vdots\\ 
x_d^1&x_d^2&\cdots & x_d^m \\
1 & 1 & \cdots & 1
\end{bmatrix}
$
\end{center}

The rank of the matrix $M(P)$ is $d+1$ since there exists a set of $d+1$ points 
in $P$ that are affinely 
independent. This 
implies that the dimension of the null space of the row space of $M(P)$ 
is $m-d-1$.  Let $\{(b_1^1, b_2^1, \ldots, b_m^1), (b_1^2, b_2^2, \ldots, 
b_m^2),$ $\ldots, (b_1^{m-d-1}, b_2^{m-d-1}, \ldots, b_m^{m-d-1}) \}$ be a 
basis of this null space. The Gale Transform of the point sequence 
$P$ is the sequence of $m$ vectors $D(P)$=$\big<(b_1^1$, $b_1^2$, $\ldots, 
b_1^{m-d-1}),$ $(b_2^1,$ $b_2^2,$ $\ldots,$ $b_2^{m-d-1}),$ $\ldots,$ $(b_m^1, 
b_m^2, 
\ldots,$ $b_m^{m-d-1})\big>$. Note that $D(P)$ can also be considered as a 
point sequence in $\mathbb{R}^{m-d-1}$ for a particular choice of the basis. 
Since the 
basis of the null space of the row space of $M(P)$ is not unique, it implies 
that the Gale Transform of $P$ is not unique. However, the following 
properties of the Gale Transform can be easily observed for any choice of 
the basis. For the sake of completeness, we give proofs for these observations.

\begin{lemma}\cite{JM}
 \label{genposi}
Every set of $m-d-1$ vectors in $D(P)$ spans $\mathbb{R}^{m-d-1}$ if the points 
in $P$ are in general position.
\end{lemma}
\begin{proof}
Without loss of generality, 
let us assume that the first $m-d-1$ vectors in $D(P)$, i.e., $(b_1^1,$ 
$b_1^2,$ 
$\ldots, $
$b_1^{m-d-1}),$ $(b_2^1,$ $b_2^2,$ $\ldots,$ $b_2^{m-d-1}),$ $\ldots,$ 
$(b_{m-d-1}^1,$ 
$b_{m-d-1}^2,$ 
$\ldots,$ $b_{m-d-1}^{m-d-1})$ do not span $\mathbb{R}^{m-d-1}$. In other 
words, 
we 
assume that these vectors are 
linearly dependent. This implies that there exist real numbers 
$\lambda_1, \lambda_2, \ldots, \lambda_{m-d-1}$, not all of them 
zero, such that $\lambda_1(b_1^1, b_1^2, \ldots, 
b_1^{m-d-1}) + 
\lambda_2(b_2^1, b_2^2, \ldots, b_2^{m-d-1}) +$ $\ldots + 
\lambda_{m-d-1}(b_{m-d-1}^1, b_{m-d-1}^2, \ldots, b_{m-d-1}^{m-d-1}) = 
\vec{0}$. Let 
us consider the vector $(\lambda_1,$ $\lambda_2,$ $\ldots,$ $\lambda_{m-d-1}, 
\lambda_{m-d}=0,$ $\ldots,$ $\lambda_m = 0)$. It is easy to see that 
$\lambda_1(b_1^1,$ $b_1^2,$ $\ldots,$ 
$b_1^{m-d-1})$ $+ \lambda_2(b_2^1,$ $b_2^2,$ $\ldots,$ $b_2^{m-d-1}) +$ $\ldots 
+ \lambda_m(b_m^1,$ $b_m^2,$ $\ldots,$ $b_m^{m-d-1})$ $= \vec{0}$ . This 
implies that the 
vector $(\lambda_1, \lambda_2, \ldots, \lambda_m)$ lies in the row space of 
$M(P)$. This further implies that there exist real numbers $\alpha_1, 
\alpha_2, \ldots, \alpha_{d+1}$, not all of them zero, such that the following 
set 
of linear equations holds for each $i$ satisfying $1 \leq i \leq m$.
$$
\alpha_1x_1^i + \alpha_2x_2^i + \ldots+ \alpha_dx_d^i + \alpha_{d+1} =\lambda_i 
$$
This implies the last $d+1$ points in $P$, i.e., $\{p_{m-d}, p_{m-d+1}, 
\ldots, p_{m}\}$ lie on the hyperplane $ 
\alpha_1x_1 + \alpha_2x_2 + \ldots+ \alpha_dx_d + \alpha_{d+1} = 0$. This is 
a contradiction to the assumption that points in $P$ are in general position in 
$\mathbb{R}^d$. This completes the proof.
\end{proof}

\begin{lemma}\cite{JM}
\label{bjection}
Consider 2 integers $u$ and $v$ satisfying $1\leq u\leq d-1$, $1\leq v\leq d-1$ 
and  
$u+v+2 = m$. If the points of $P$ are in general 
position in $\mathbb{R}^d$, there exists a 
bijection between the crossing pairs of $u$ and $v$-simplices
in $P$ and linear separations of $D(P)$ into $D(P_1)$ and $D(P_2)$ 
such that $|D(P_1)| = u+1$ and $|D(P_2)| = v+1$.  
\end{lemma}

\begin{proof}
$(\Rightarrow)$ Let $\sigma$ be a $u$-simplex that crosses a 
$v$-simplex $\nu$, such that $1\leq u\leq d-1$, $1\leq v\leq d-1$ and $u+v+2 
=m$. Without loss of generality, we assume that $\sigma$ is spanned by the 
first $u+1$ points $\{p_1, p_2, \ldots p_{u+1}\}$ of $P$ and $\nu$ is spanned 
by the remaining $v+1$ points $\{p_{u+2}, p_{u+3}, \ldots, p_m\}$ of $P$. As 
there exists a crossing between $\sigma$ and $\nu$, we know that there exists a 
point $p$ 
belonging to the relative interiors of both $\sigma$ and $\nu$. This implies 
that there exist real numbers $\lambda_k \geq 0$, $1 \leq k \leq m$, 
satisfying the following equations:
\begin{equation*}
 p = \sum\limits_{i \in \{1, 2, \ldots, u+1 \}} \lambda_ip_i  = 
\sum\limits_{j 
\in \{u+2, u+3, \ldots, m \}} \lambda_jp_j 
\end{equation*}
\begin{equation*}
\sum\limits_{i \in \{1, 2, \ldots, u+1 \}} \lambda_i  = \sum\limits_{j \in 
\{u+2, u+3, \ldots, m \}} \lambda_j = 1
\end{equation*}

Therefore, we obtain the following equation.

\begin{center}
\begin{equation}\label{eqn:glprop}
 \begin{bmatrix}
x_1^1&x_1^2&\cdots & x_1^m  \\
x_2^1&x_2^2&\cdots & x_2^m \\
\vdots & \vdots & \vdots & \vdots\\ 
x_d^1&x_d^2&\cdots & x_d^m \\
1 & 1 & \cdots & 1
 \end{bmatrix}
     \begin{bmatrix}
      \lambda_1 \\
        \vdots\\
        \lambda_{u+1}\\
        -\lambda_{u+2}\\
        \vdots\\
       -\lambda_m\\
  \end{bmatrix}
  =
  \begin{bmatrix}
       0 \\
      0 \\
      \vdots\\
      0
  \end{bmatrix}
\end{equation}
\end{center}

It is evident from Equation 
\ref{eqn:glprop} that the vector $(\lambda_1,$ $
\lambda_2,$ $\ldots,$ $\lambda_{u+1},$ $-\lambda_{u+2},$ $\ldots,$ 
$-\lambda_m)$ lies in 
the null space of the row space of $M(P)$. This implies that 
$(\lambda_1,$ $\lambda_2,$ $\ldots,$ 
$\lambda_{u+1},$ $ -\lambda_{u+2},$ $\ldots, -\lambda_m)$ $= \alpha_1(b_1^1, 
b_2^1,$ $ \ldots, b_m^1) + \alpha_2(b_1^2,$ $b_2^2,$ $ \ldots, b_m^2)$ $+ 
\ldots 
+ \alpha_{m-d-1}$ $(b_1^{m-d-1}$, 
$b_2^{m-d-1},$ $\ldots,$ $b_m^{m-d-1})$, for some real numbers $\alpha_1, 
\alpha_2, \ldots, 
\alpha_{m-d-1}$, not all of them zero. In other words, there exist $\alpha_1,$ 
$\alpha_2,$ $\ldots, \alpha_{m-d-1}$, not all of them zero, such that 
$\alpha_1b_i^1 + \alpha_2b_i^2 +$ $\ldots +$ 
$\alpha_{m-d-1}$ $b_i^{m-d-1}  > 0$ for $i= 1, 2, \ldots, u+1$, and 
$\alpha_1b_j^1 + 
\alpha_2b_j^2 + \ldots + \alpha_{m-d-1}b_j^{m-d-1} < 0$ for $j= u+2, u+3, 
\ldots, m$. This shows that the hyperplane $\sum\limits_{i \in \{1, 2, 
\ldots,m-d-1\}}\alpha_ix_i = 0$ separates the first $u+1$ vectors in $D(P)$ 
from 
the remaining $v+1$ vectors.\\

($\Leftarrow$) 
Without loss of generality, let us assume that the hyperplane 
$$\sum\limits_{i \in \{1, \ldots,m-d-1\}} \alpha'_ix_i = 0$$ separates the 
first 
$u+1$ vectors in $D(P)$ from 
the 
remaining $v+1$ vectors. This implies that there exists a vector $(\mu'_1, 
\mu'_2, \ldots, \mu'_m)= \alpha'_1(b_1^1, b_2^1, \ldots, b_m^1) + 
\alpha'_2(b_1^2,$ $b_2^2,$ $\ldots,$ $b_m^2)$ $+ \ldots +$ $
\alpha'_{m-d-1}(b_1^{m-d-1},$ 
$b_2^{m-d-1},$ $\ldots,$ $b_m^{m-d-1})$ such that the signs of $\mu'_i$ for $1 
\leq i \leq u+1$ are opposite to the signs of $\mu'_j$ for $u+2 \leq j 
\leq m$. Without loss of generality, let us assume that $\mu'_i > 0$ for $1 
\leq i \leq u+1$ and $\mu'_j < 0$ for $u+2 \leq j 
\leq m$.  As this vector $(\mu'_1, 
\mu'_2, \ldots, \mu'_m)$ lies in the null space of the row space of 
$M(P)$, it satisfies the following equation.
\vspace{-0.5cm}
\begin{center}
\begin{equation}\label{eqn:glprop1}
 \begin{bmatrix}
x_1^1&x_1^2&\cdots & x_1^m  \\
x_2^1&x_2^2&\cdots & x_2^m \\
\vdots & \vdots & \vdots & \vdots\\ 
x_d^1&x_d^2&\cdots & x_d^m \\
1 & 1 & \cdots & 1
 \end{bmatrix}
     \begin{bmatrix}
      \mu'_1 \\
       \mu'_2 \\
        \vdots\\
       \mu'_m\\
  \end{bmatrix}
  =
  \begin{bmatrix}
       0 \\
      0 \\
      \vdots\\
      0
  \end{bmatrix}
\end{equation}
\end{center}
From Equation \ref{eqn:glprop1}, we obtain the following.
\begin{equation*}
\sum\limits_{i \in \{1, 2, \ldots, u+1 \}} \mu'_ip_i  = 
\sum\limits_{j 
\in \{u+2, u+3, \ldots, m \}} -\mu'_jp_j 
\end{equation*}
\begin{equation*}
\sum\limits_{i \in \{1, 2, \ldots, u+1 \}} \mu'_i  = \sum\limits_{j \in 
\{u+2, u+3, \ldots, m \}} -\mu'_j 
\end{equation*}
Rearranging the above equations, we obtain the following.
\begin{equation*}
\sum\limits_{i \in \{1, 2, \ldots, u+1 \}} {\dfrac{\mu'_i}{\sum\limits_{k \in 
\{1, 2, \ldots, u+1 \}} \mu'_k}}p_i  = 
\sum\limits_{j 
\in \{u+2, u+3, \ldots, m \}} {\dfrac{\mu'_j}{\sum\limits_{k \in \{u+2, u+3, 
\ldots, m \}} \mu'_k}}p_j 
\end{equation*}

\begin{equation*}
 \sum\limits_{i \in \{1, 2, \ldots, u+1 \}} {\dfrac{\mu'_i}{\sum\limits_{k \in 
\{1, 2, \ldots, u+1 \}} \mu'_k}}  = \sum\limits_{j \in 
\{u+2, u+3, \ldots, m \}} {\dfrac{\mu'_j}{\sum\limits_{k \in \{u+2, u+3, 
\ldots, 
m \}} \mu'_k}} = 1
\end{equation*}

It shows that there exists a crossing between the $u$-simplex spanned by 
the first $u+1$ points of $P$ and the $v$-simplex spanned by the remaining 
$v+1$ points of $P$.            
\end{proof}

\noindent An argument similar to the one used above can be used to prove the following Lemma.

\begin{lemma}\cite{JM}
\label{convexity}
The points in $P$ are in convex position in $\mathbb{R}^d$ if 
and only if there 
is 
no linear hyperplane $h$ with exactly one vector from $D(P)$ on one side of $h$.
\end{lemma}

In the following, we derive the Gale Transform of a sequence of $m \geq 
d+1$ points placed on the $d$-dimensional moment curve. Let this point sequence 
be $A=\big<(t_1, (t_1)^2,$ $\ldots, (t_1)^d),$ $(t_2, (t_2)^2, \ldots, 
(t_2)^d),$ $\ldots,$ $
(t_m, (t_m)^2,$ $ \ldots, (t_m)^d)\big>$, where each $t_i$ for $1 \leq i 
\leq m$ is a real number satisfying $t_i < t_j$ for $i < j$.
We obtain the following two Lemmas for $m = d+3$ and $m = 2d$, respectively.
\begin{lemma}
\label{lem:d3g}
The following sequence of $2$-dimensional vectors $D(A)=$ $\big<v_1,$ $v_2,$ 
$\ldots,$ $v_{d+3}\big>$ can be obtained by the Gale Transform of $A=$ 
$\big<(t_1, (t_1)^2,$ $\ldots, (t_1)^d),$ 
$(t_2,$ $(t_2)^2,$ $\ldots,$ 
$(t_2)^d),$ $\ldots,$ $
(t_{d+3},$ $(t_{d+3})^2,$ $ \ldots,$ $(t_{d+3})^d)\big>$.\\\\
\scalebox{.8}{
$v_i = 
\begin{cases}
\Bigg(
(-1)^{d+1}{\dfrac{\prod\limits_{j\in{\{1,2,\cdots,d+1\}\setminus 
\{i\}}} 
(t_{d+2}-t_j)}{\prod\limits_{k\in{\{1,2,\cdots,d+1\}\setminus \{i\}}} 
(t_k-t_i)}}, (-1)^{d+1}{\dfrac{\prod\limits_{j\in{\{1,2,\cdots,d+1\}\setminus 
\{i\}}} 
(t_{d+3}-t_j)}{\prod\limits_{k\in{\{1,2,\cdots,d+1\}\setminus \{i\}}} 
(t_k-t_i)}}\Bigg) &  if~~ i \in \{ 1,2,\cdots,d+1\} \\
(1, 0) &  if~~ i= d+2\\
(0, 1) &  if~~ i= d+3\\
\end{cases} 
$
}
\end{lemma}

\begin{proof}
Let us consider the following matrix M(A).
 \begin{center}
$
M(A) = 
\begin{bmatrix}
t_1&t_2&\cdots & t_{d+3}  \\
(t_1)^2&(t_2)^2&\cdots & (t_{d+3})^2 \\
\vdots & \vdots & \vdots & \vdots\\ 
(t_1)^d&(t_2)^d&\cdots & (t_{d+3})^d \\
1 & 1 & \cdots & 1
\end{bmatrix}
$
\end{center}
To obtain the basis of the null space, we need to find solutions of the 
following $d+1$ equations involving $d+3$ variables $\gamma_1,$ $\gamma_2, 
\ldots,$ $\gamma_{d+3}$.
\vspace{-0.5cm}

\begin{center}\begin{equation}\label{eqn:mmntd3}
 \begin{bmatrix}
t_1&t_2&\cdots & t_{d+3}  \\
(t_1)^2&(t_2)^2&\cdots & (t_{d+3})^2 \\
\vdots & \vdots & \vdots & \vdots\\ 
(t_1)^d&(t_2)^d&\cdots & (t_{d+3})^d \\
1 & 1 & \cdots & 1
 \end{bmatrix}
     \begin{bmatrix}
      \gamma_1 \\
       \gamma_2 \\
        \vdots\\
       \gamma_{d+3}\\
  \end{bmatrix}
  =
  \begin{bmatrix}
       0 \\
      0 \\
      \vdots\\
      0
  \end{bmatrix}
\end{equation}
\end{center}
Rearranging Equation \ref{eqn:mmntd3}, we get the following:
\begin{center}$
 \begin{bmatrix}
      \gamma_1 \\
       \gamma_2 \\
        \vdots\\
       \gamma_{d+1}\\
  \end{bmatrix} =
  -
  \begin{bmatrix}
  t_1&t_2&\cdots & t_{d+1}  \\
(t_1)^2&(t_2)^2&\cdots & (t_{d+1})^2 \\
\vdots & \vdots & \vdots & \vdots\\ 
(t_1)^d&(t_2)^d&\cdots & (t_{d+1})^d \\
1 & 1 & \cdots & 1
\end{bmatrix}^{-1}
\begin{bmatrix}
 t_{d+2}&t_{d+3}  \\
(t_{d+2})^2&(t_{d+3})^2 \\
\vdots & \vdots \\ 
(t_{d+2})^d&(t_{d+3})^d& \\
1 & 1 
\end{bmatrix}
\begin{bmatrix}
 \gamma_{d+2} \\
  \gamma_{d+3} \\
\end{bmatrix}$
\end{center}
 
Setting $\gamma_{d+2} = 1$ and $\gamma_{d+3} = 0$, we obtain the following for 
every $i$ satisfying $1\leq i\leq d+1$.
\begin{center}
$\gamma_i =
(-1)^{d+1}{\dfrac{\prod\limits_{j\in{\{1,2,\cdots,d+1\}\setminus 
\{i\}}} 
(t_{d+2}-t_j)}{\prod\limits_{k\in{\{1,2,\cdots,d+1\}\setminus \{i\}}} 
(t_k-t_i)}}$.\\
\end{center}
Setting $\gamma_{d+2} = 0$ and $\gamma_{d+3} = 1$, we obtain the 
following for every $i$ satisfying $1\leq i\leq d+1$.
\begin{center}
$\gamma_i =
(-1)^{d+1}{\dfrac{\prod\limits_{j\in{\{1,2,\cdots,d+1\}\setminus 
\{i\}}} 
(t_{d+3}-t_j)}{\prod\limits_{k\in{\{1,2,\cdots,d+1\}\setminus \{i\}}} 
(t_k-t_i)}}$. 
\end{center}
Note that the vectors 
\scalebox{.9}{$
\Bigg(
(-1)^{d+1}{\dfrac{\prod\limits_{j\in{\{1,2,\cdots,d+1\}\setminus 
\{1\}}} 
(t_{d+2}-t_j)}{\prod\limits_{k\in{\{1,2,\cdots,d+1\}\setminus \{1\}}} 
(t_k-t_{1})}},$
$\ldots,$
$(-1)^{d+1}{\dfrac{\prod\limits_{j\in{\{1,2,\cdots,d+1\}\setminus 
\{d+1\}}} 
(t_{d+2}-t_j)}{\prod\limits_{k\in{\{1,2,\cdots,d+1\}\setminus \{d+1\}}} 
(t_k-t_{d+1})}},$
$1,$
$0\Bigg)$} and 
\scalebox{.8}{
$\Bigg(
(-1)^{d+1}{\dfrac{\prod\limits_{j\in{\{1,2,\cdots,d+1\}\setminus 
\{1\}}}
(t_{d+3}-t_j)}{\prod\limits_{k\in{\{1,2,\cdots,d+1\}\setminus \{1\}}} 
(t_k-t_1)}},$
$\ldots,$
$(-1)^{d+1}{\dfrac{\prod\limits_{j\in{\{1,2,\cdots,d+1\}\setminus 
\{d+1\}}} 
(t_{d+3}-t_j)}{\prod\limits_{k\in{\{1,2,\cdots,d+1\}\setminus \{d+1\}}} 
(t_k-t_{d+1})}},$
$0,$
$1\Bigg)$
} are linearly independent and form a basis of the 
null space of the row space of M(A). Hence, the result follows.
\end{proof}

\begin{lemma}
\label{lem:2dmg}
The following sequence of $(d-1)$-dimensional vectors $D(A)=$ $\big<v_1, v_2, 
$\ldots,$ v_{2d} \big>$ can be obtained by the Gale Transform of $A =$ 
$\big<(t_1, (t_1)^2,$ $\ldots,$ $(t_1)^d),$ $(t_2, 
(t_2)^2, \ldots, 
(t_2)^d),$ $\ldots,$ $
(t_{2d}, (t_{2d})^2,$ $ \ldots, (t_{2d})^d)\big>$.\\

\scalebox{.7}{$v_{i} = 
\begin{cases}
\Bigg(
(-1)^{d+1}{\dfrac{\prod\limits_{j\in{\{1,2,\cdots,d+1\}\setminus 
\{i\}}} 
(t_{d+2}-t_j)}{\prod\limits_{k\in{\{1,2,\cdots,d+1\}\setminus \{i\}}} 
(t_k-t_i)}}, (-1)^{d+1}{\dfrac{\prod\limits_{j\in{\{1,2,\cdots,d+1\}\setminus 
\{i\}}} 
(t_{d+3}-t_j)}{\prod\limits_{k\in{\{1,2,\cdots,d+1\}\setminus \{i\}}} 
(t_k-t_i)}},
\ldots,
(-1)^{d+1}{\dfrac{\prod\limits_{j\in{\{1,2,\cdots,d+1\}\setminus 
\{i\}}} 
(t_{2d}-t_j)}{\prod\limits_{k\in{\{1,2,\cdots,d+1\}\setminus \{i\}}} 
(t_k-t_i)}}\Bigg)\\  
~~~~~~~~~~~~~~~~~~~~~~~~~~~~~~~~~~~~~~~~~~~~~~~~~~~~~~~~~~~~~~~~~~~~~~~~~~~~~~~~
~~~~~~~~~~~~~~~~~~~~~~~~~~~~~~~~~ if~~ i \in \{ 1,2,\cdots,d+1\}\ \\
(1, 0, \ldots, 0)
~~~~~~~~~~~~~~~~~~~~~~~~~~~~~~~~~~~~~~~~~~~~~~~~~~~~~~~~~~~~~~~~~~~~~~~~~~~~~~~~
~~~~~~~~~~~~~~~~~~~ if~~ i= d+2\\
(0, 1,\ldots,
0)~~~~~~~~~~~~~~~~~~~~~~~~~~~~~~~~~~~~~~~~~~~~~~~~~~~~~~~~~~~~~~~~~~~~~~~~~~~~~~
 
~~~~~~~~~~~~~~~~~~~~~ if~~ i= d+3\\
\vdots 
~~~~~~~~~~~~~~~~~~~~~~~~~~~~~~~~~~~~~~~~~~~~~~~~~~~~~~~~~~~~~~~~~~~~~~~~~~~~~~~~
~~~~~~~~~~~~~~~~~~~~~~~~~~~~~~~~~~~~~~~ \vdots\\
(0, 0, \ldots, 
1)~~~~~~~~~~~~~~~~~~~~~~~~~~~~~~~~~~~~~~~~~~~~~~~~~~~~~~~~~~~~~~~~~~~~~~~~~~~~~~
 
~~~~~~~~~~~~~~~~~~~~~ if~~ i= 2d\\
\end{cases} 
$}
\end{lemma}

\begin{proof}
Let us consider the matrix $M(A)$. 
\begin{center}
$
M(A) = 
\begin{bmatrix}
t_1&t_2&\cdots & t_{2d}  \\
(t_1)^2&(t_2)^2&\cdots & (t_{2d})^2 \\
\vdots & \vdots & \vdots & \vdots\\ 
(t_1)^d&(t_2)^d&\cdots & (t_{2d})^d \\
1 & 1 & \cdots & 1
\end{bmatrix}
$
\end{center}

\noindent To compute the basis of this null space, we solve the following 
equation.
\vspace{-0.5cm}
\begin{center}
\begin{equation}\label{eqn:mmnt2d}
 \begin{bmatrix}
t_1&t_2&\cdots & t_{2d}  \\
(t_1)^2&(t_2)^2&\cdots & (t_{2d})^2 \\
\vdots & \vdots & \vdots & \vdots\\ 
(t_1)^d&(t_2)^d&\cdots & (t_{2d})^d \\
1 & 1 & \cdots & 1
 \end{bmatrix}
     \begin{bmatrix}
      \gamma_1 \\
       \gamma_2 \\
        \vdots\\
       \gamma_{2d}\\
  \end{bmatrix}
  =
  \begin{bmatrix}
       0 \\
      0 \\
      \vdots\\
      0
  \end{bmatrix}
\end{equation}
\end{center}
Rearranging Equation \ref{eqn:mmnt2d}, we get the following:\\

\scalebox{0.9}
{$
 \begin{bmatrix}
      \gamma_1 \\
       \gamma_2 \\
        \vdots\\
       \gamma_{d+1}\\
  \end{bmatrix} =
  -\begin{bmatrix}
  t_1&t_2&\cdots & t_{d+1}  \\
(t_1)^2&(t_2)^2&\cdots & (t_{d+1})^2 \\
\vdots & \vdots & \vdots & \vdots\\ 
(t_1)^d&(t_2)^d&\cdots & (t_{d+1})^d \\
1 & 1 & \cdots & 1
\end{bmatrix}^{-1}
\begin{bmatrix}
 t_{d+2}&t_{d+3}& \cdots& t_{2d}    \\
(t_{d+2})^2&(t_{d+3})^2 &\cdots& (t_{2d})^2 \\
\vdots & \vdots & \vdots & \vdots \\ 
(t_{d+2})^d&(t_{d+3})^d &\cdots& (t_{2d})^d \\
1 & 1 & 1 & 1
\end{bmatrix}
\begin{bmatrix}
 \gamma_{d+2} \\
  \gamma_{d+3} \\
  \vdots\\
  \gamma_{2d}
\end{bmatrix}
$
}\\

Consider an $r$ satisfying $1 \leq r\leq d-1$. For this $r$, we obtain the 
values of $\gamma_1$, $\gamma_2$, $\ldots $, $\gamma_{d+1}$ by 
setting $(\gamma_{d+2}, \gamma_{d+3}, ... , \gamma_{2d}) = e_r$, where $e_r$ is 
the $r^{th}$ row of the identity matrix $I_{d-1}$.

\begin{center}
$\gamma_{i} =$
$(-1)^{d+1}{\dfrac{\prod\limits_{j\in{\{1,2,\cdots,d+1\}\setminus 
\{i\}}} 
(t_{d+1+r}-t_j)}{\prod\limits_{k\in{\{1,2,\cdots,d+1\}\setminus \{i\}}} 
(t_k-t_i)}}$ 
\end{center}
For $d-1$ different values of $r$, the $d-1$ vectors obtained in the 
above-mentioned way form a basis of the 
null space of the row space of M(A). Hence, the result follows.
\end{proof}
\section{Lower Bound on the $d$-Dimensional Rectilinear Crossing 
Number of $K^d_{2d}$}
\label{gencase}
In this Section, we improve the lower bound on $c_d$ by using the Proper 
Separation 
Theorem and Lemma \ref{extension} mentioned below. Recall that $c_d$ denotes 
the $d$-dimensional rectilinear crossing number of $K^d_{2d}$. Let $P$ denote 
the set of $2d$ vertices of $K_{2d}^d$ that are in general position in 
$\mathbb{R}^d$. Two nonempty convex sets $C$ and $D$ in $\mathbb{R}^d$ are said 
to be 
\textit{properly separated} if there exists a $(d-1)$-dimensional hyperplane 
$h$ such 
that $C$ and $D$ lie in the opposite closed half-spaces determined by $h$, and 
$C$ and $D$ are not both contained in the hyperplane $h$ \cite{OG}.

\vspace{.1cm}

\noindent{\textbf{Proper Separation Theorem.}~\cite{OG}} 
\textit{Two nonempty convex 
sets $C$ and $D$ in $\mathbb{R}^d$ can be properly separated if and only if 
their 
relative interiors are disjoint.}
\begin{lemma}
\label{extension} 
Consider two disjoint point sets $U, V \subset P$ such that $|U|= p$, $|V|= 
q$, $2 \leq p, q \leq d$ and $p+q \geq d+1$. If the $(p-1)$-simplex formed by 
$U$ crosses the 
$(q-1)$-simplex formed by $V$, then the $(d-1)$-simplices formed by any 
two disjoint point sets $U'\supseteq U$ and $V' \supseteq V$ satisfying $|U'| = 
|V'| =d$ also form a crossing pair.
\end{lemma}
\begin{proof}
For the given disjoint sets $U= \{x_1, x_2, \ldots, x_p\}$ and $V=\{y_1, 
y_2, \ldots, y_q\}$, assume for the sake of contradiction 
that there exist two $(d-1)$-simplices, formed respectively by the disjoint 
point sets 
$U^\prime \supseteq U$ and $V^\prime \supseteq V$, that do not cross. If 
$Conv(U^\prime) \cap Conv(V^\prime) = \phi$, then it is easy to observe that
$Conv(U) \cap Conv(V) = \phi$. This leads to a contradiction. Otherwise, 
consider the convex sets $Conv(U^\prime)$ and $Conv(V^\prime)$. Since 
$Conv(U^\prime)$ and $Conv(V^\prime)$ do not cross, their relative 
interiors are disjoint. The Proper Separation Theorem guarantees 
that there exists a $(d-1)$-dimensional hyperplane h such that 
$Conv(U^\prime)$ and $Conv(V^\prime)$ lie in the opposite closed half-spaces 
determined by $h$, which further implies that $Conv(U)$ and $Conv(V)$ lie in 
the opposite closed half-spaces determined by $h$. As 
$Conv(U)$ and $Conv(V)$ have a common point in their relative interiors, 
$U$ and $V$ cannot be properly separated by $h$. It implies that all the 
$p+q$ points of $U \cup V$ lie on the $(d-1)$-dimensional hyperplane $h$. Since 
the $p+q \geq d+1$ points in $U \cup V$ are in general position, this leads to 
a contradiction.
\end{proof}

\noindent\textbf{Proof of Theorem \ref{thm11}:}
Consider the hypergraph $K^d_{2d}$ whose vertices are in general position in 
$\mathbb{R}^d$, and let $A$ be any subset of $d+3$ vertices selected from these 
vertices. The Gale Transform $D(A)$ of the point set $A$ 
contains $d+3$ vectors in $\mathbb{R}^2$, which can also be considered as a 
sequence 
of $d+3$ points (as mentioned in Section \ref{tused}). In order to apply the 
Ham-Sandwich Theorem (mentioned in the Introduction) in $\mathbb{R}^2$, we 
assign the 
points in $D(A)$ to $P_1$ and the origin to $P_2$ to obtain a line $l$ passing 
through the origin that bisects the points in $D(A)$ such that each 
partition (open half-space) contains at most $\floor{\frac{1}{2}|D(A)|}$ points 
from $D(A)$. 
Since the points in $A$ are in general position, every pair of vectors in 
$D(A)$ 
spans $\mathbb{R}^2$. Hence, at most one point from $D(A)$ can lie on $l$. 
As a consequence, $l$ can be rotated using the origin as the axis of rotation 
to obtain a proper linear 
separation of $D(A)$ into 2 subsets 
$l_1^+$ and $l_1^-$ of size $\floor{\frac{d+3}{2}}$ and 
$\ceil{\frac{d+3}{2}}$, respectively, such that $l_1^+$ denotes the positive 
(counter-clockwise) side and $l_1^-$ denotes the negative (clockwise) side of 
$l$. Lemma~\ref{bjection} implies that this proper linear separation 
corresponds to a crossing pair of a 
$(\floor{\frac{d+3}{2}}-1)$-simplex and a $(\ceil{\frac{d+3}{2}}-1)$-simplex in 
$\mathbb{R}^d$. We observe from Lemma~\ref{extension} that this crossing pair 
of simplices can be used to obtain ${d-3\choose \floor{\frac{d-3}{2}}}$ 
distinct crossing pairs of $(d-1)$-simplices formed by the vertices of the 
hypergraph $K_{2d}^d$. 

We rotate $l$ clockwise using the origin as the axis of 
rotation, until one of the $d+3$ points in $D(A)$ moves from one side of the 
line $l$ to the other side. Since every pair of vectors in $D(A)$ spans 
$\mathbb{R}^2$, it 
can be observed that exactly one point of $D(A)$ can change its side at any 
particular time during the rotation of $l$. We further rotate $l$ clockwise to 
obtain another new partition $\{l_2^+,l_2^-\}$, each having at least 
$\floor{\frac{d+1}{2}}$ points, at the instance a point in either $l_1^+$ or 
$l_1^-$ changes its side. This new linear separation corresponds to a crossing 
pair of simplices 
in $\mathbb{R}^d$, which can be used to obtain at least ${d-3 
\choose\floor{\frac{d-5}{2}}}$ distinct crossing 
pairs of $(d-1)$-simplices formed by the vertices of the hypergraph $K_{2d}^d$. 
Note that all the crossing pairs of simplices obtained by extending the 
partitions $\{l_1^+,l_1^-\}$  and $\{l_2^+,l_2^-\}$ are distinct.
Continuing in this manner for any $1 \leq k \leq \floor{\frac{d-3}{2}}-1$, we 
rotate $l$ clockwise to obtain a new partition $\{l_{k+1}^+$, $l_{k+1}^-\}$, 
each having at least 
$\floor{\frac{d-2k+3}{2}}$ points, at any time a point in either $l_{k}^+$ or 
$l_{k}^-$ changes its side. Therefore, the corresponding crossing pair of 
simplices in $\mathbb{R}^d$ can be extended to crossing pairs of 
$(d-1)$-simplices in 
at least ${d-3 \choose {d-\floor{\frac{d-2k+3}{2}}}}=$ ${d-3 \choose 
\floor{\frac{d-2k-3}{2}}}$ distinct ways. Hence, the number of crossing 
pairs of $(d-1)$-simplices obtained using this method is at least
$${d-3 \choose \floor{\frac{d-3}{2}}} + {d-3 \choose \floor{\frac{d-5}{2}}} + 
{d-3\choose \floor{\frac{d-7}{2}}} + ... + {d-3 \choose 1} = \Theta (2^d). 
$$

\section{Number of Crossing Pairs of Hyperedges of 
$K^d_{2d}$ on the Moment Curve}
\label{cncp}

In this Section, we obtain the value of $c_d^m$. Recall that $c_d^m$ 
denotes the number of crossing pairs of hyperedges of $K^d_{2d}$, when all the 
$2d$ vertices of $K^d_{2d}$ are placed on the $d$-dimensional moment curve. We 
first prove a lower bound on $c_d^m$ using an approach similar to the proof of 
Theorem \ref{thm11} and show later that this bound can be improved by using other
techniques to obtain the exact value of $c_d^m$. Let 
$A$ be a subset of any $d+3$ vertices selected from these $2d$ vertices. To 
establish a lower bound on $c_d^m$, we count the number of linear separations 
of 
the vectors in $D(A)$, i.e., the Gale Transform of $A$, obtained in 
Lemma~\ref{lem:d3g}. For the given $A=$ $\big<(t_1, (t_1)^2,$ 
$\ldots,(t_1)^d),$ 
$(t_2, (t_2)^2, \ldots, 
(t_2)^d),$ $\ldots,$ $
(t_{d+3}, (t_{d+3})^2,$ $ \ldots, (t_{d+3})^d)\big>$, where $t_1 < t_2 < 
\ldots < t_{d+3}$, the $i^{th}$ vector $v_i$ in $D(A)$ is the following:\\\\
\scalebox{0.85}
{$v_i = 
\begin{cases}
\Bigg(
(-1)^{d+1}{\dfrac{\prod\limits_{j\in{\{1,2,\cdots,d+1\}\setminus 
\{i\}}} 
(t_{d+2}-t_j)}{\prod\limits_{k\in{\{1,2,\cdots,d+1\}\setminus \{i\}}} 
(t_k-t_i)}}, (-1)^{d+1}{\dfrac{\prod\limits_{j\in{\{1,2,\cdots,d+1\}\setminus 
\{i\}}} 
(t_{d+3}-t_j)}{\prod\limits_{k\in{\{1,2,\cdots,d+1\}\setminus \{i\}}} 
(t_k-t_i)}}\Bigg) &  if~~ i \in \{ 1,2,\cdots,d+1\} \\
(1, 0) &  if~~ i= d+2\\
(0, 1) &  if~~ i= d+3\\
\end{cases} 
$}\\ 
Note that for every 
$1\leq i \leq d+3$, each vector $v_i$ in $D(A)$ is represented as an ordered 
pair $(a_i, b_i)$ where $a_i,b_i \in\mathbb{R}$. We denote the slope of the 
vector 
$v_i$ as 
$s_i=\dfrac{b_i}{a_i}$. In order to count the number of linear separations, we 
observe the following properties of these vectors.
\begin{observation}
\label{obs:diffr}
The sequence of $2$-dimensional vectors $D(A)=$ $\big<v_1, v_2,$ 
$\ldots,$ $v_{d+3}\big>$ having slopes $\big<s_1, s_2, \ldots, s_{d+3}\big>$  
satisfies the following properties.
\begin{enumerate}
  \item[(i)] For any $1 \leq i \leq d+1$, $v_i$ lies in the first (third) 
quadrant if $d+1+i$ is odd (even).
  \item[(ii)] $\infty =s_{d+3}> s_1 > s_2 > \ldots > s_{d+1} > s_{d+2}= 0$.
\end{enumerate}
\end{observation}

\begin{lemma}
 \label{lem:upmomnt}
 $c_d^m= \Omega(2^d\sqrt{d})$.
\end{lemma}

\begin{proof}
Consider the vectors in $D(A)$, that can also be considered as a sequence of 
$d+3$ points in $\mathbb{R}^2$. We apply the Ham-Sandwich Theorem by assigning 
the 
points in $D(A)$ to $P_1$ and the origin to $P_2$ to obtain a line $l$ passing 
through the origin that bisects the points in $D(A)$ into two 
partitions, each containing at most $\floor{\frac{1}{2}|D(A)|}$ points.
Since at most one 
point from $D(A)$ can lie on $l$, it can 
be rotated using the origin as the axis of rotation to obtain a proper linear 
separation of $D(A)$ into 2 subsets 
$l_1^+$ (positive or counter-clockwise side) and $l_1^-$ (negative or 
clockwise side) of size $\floor{\frac{d+3}{2}}$ and $\ceil{\frac{d+3}{2}}$, 
respectively. This proper linear separation corresponds to a crossing pair of a 
$(\floor{\frac{d+3}{2}}-1)$-simplex and a $(\ceil{\frac{d+3}{2}}-1)$-simplex in 
$\mathbb{R}^d$, as shown in Lemma~\ref{bjection}. It follows from 
Lemma~\ref{extension} that this crossing pair of simplices can be used to 
obtain 
${d-3\choose \floor{\frac{d-3}{2}}}$ distinct crossing pairs of 
$(d-1)$-simplices formed by the vertices of the hypergraph $K_{2d}^d$. We 
rotate 
$l$ clockwise using the origin as the axis of rotation until one of the $d+3$ 
points in $D(A)$ moves from one side of the line to the other side to obtain 
new 
subsets $\{l_2^+,l_2^-\}$, each having at least $\floor{\frac{d+1}{2}}$ points. 
This new linear separation $\{l_2^+,l_2^-\}$ corresponds to a crossing pair of 
simplices 
in $\mathbb{R}^d$, which can be used to obtain at least 
${d-3 \choose\floor{\frac{d-5}{2}}}$ distinct crossing 
pairs of $(d-1)$-simplices formed by the vertices of the hypergraph $K_{2d}^d$. 
Note that all the crossing pairs of simplices obtained by extending the 
partitions $\{l_1^+,l_1^-\}$  and $\{l_2^+,l_2^-\}$ are distinct. Since all the 
$d + 3$ points of $A$ lie on the $d$-dimensional moment curve, Observation 
\ref{obs:diffr} implies that the 
sequence of vectors in $D(A)$, excluding $v_{d+2}$ and 
$v_{d+3}$, lie alternatively in the first and third quadrants with increasing 
slopes. As a consequence, another clockwise 
rotation of $l$ results in a point in $D(A)$ changing its 
side at some point of time from a side having more than or equal 
to $\ceil{\frac{d+3}{2}}$ points to the other side. 
This creates a new partition $\{l_{3}^+$, $l_{3}^-\}$, each containing at least 
$\floor{\frac{d+3}{2}}$ points. We continue rotating $l$ clockwise until we 
obtain the partition $\{l_1^+, l_1^-\}$ again. In this way, we obtain at least 
$2\floor{\frac{d+3}{2}}$ distinct partitions of $D(A)$ such that each subset in 
a partition contains at least $\floor{\frac{d+1}{2}}$ points. Hence, the number 
of crossing pairs of hyperedges spanned by the vertices of $K^d_{2d}$ placed on 
the $d$-dimensional moment curve is at least
$$2\floor{\frac{d+3}{2}}{d-3 \choose \ceil{\frac{d-5}{2}}} = 
\Theta (2^d \sqrt{d}).$$      
\end{proof}
However, we show below that this lower bound is far from being optimal. In 
fact, we use Lemma \ref{lem:dpach} and Lemma \ref{lem:dpach1} to prove Theorem 
\ref{momntup1} that implies 
$c_d^m= 
\Theta(\frac{4^d}{\sqrt{d}})$. Let us define the ordering between two points $p 
=\{t, (t)^2, \ldots, (t)^d\}$ and $p' = \{t', (t')^2, \ldots, (t')^d\}$ on the 
$d$-dimensional moment curve by $p \prec p'$ 
($p$ precedes $p'$) if $t < t'$. 
\begin{lemma}\cite{DP} \label{lem:dpach}
 Let $p_1 \prec p_2 \prec 
\ldots \prec p_{\floor{\frac{d}{2}}+1}$ and $q_1 \prec q_2 \prec \ldots \prec 
q_{\ceil{\frac{d}{2}}+1}$ be two 
distinct point sequences on the $d$-dimensional moment curve such that $p_i 
\neq q_j$ for any $1 \leq i \leq \floor{\frac{d}{2}}+1$ and $1 \leq j \leq 
\ceil{\frac{d}{2}}+1$. The 
$\floor{\frac{d}{2}}$-simplex and the $\ceil{\frac{d}{2}}$-simplex, 
formed respectively by these point sequences, cross if and only if every 
interval 
$(q_j , q_{j+1})$ contains exactly one $p_i$ and every interval $(p_i , 
p_{i+1})$ 
contains exactly one $q_j$.  
\end{lemma}

\begin{lemma}\cite{DP} \label{lem:dpach1}
Let $P$ and $Q$ be two vertex-disjoint $(d-1)$-simplices such that each of the 
$2d$ 
vertices belonging to these simplices lies on the $d$-dimensional moment curve. 
If $P$ and $Q$ cross, then there exist a $\floor{\frac{d}{2}}$-simplex $U 
\subsetneq P$ and  another 
$\ceil{\frac{d}{2}}$-simplex $V \subsetneq Q$ such that $U$ and $V$ cross. 
\end{lemma}

\noindent\textbf{Proof of Theorem \ref{momntup1}:} Let $\{C,D\}$ be a pair of 
disjoint vertex sets, each having $d$ vertices of $K^d_{2d}$ placed on the 
$d$-dimensional moment curve $\gamma = \{(t,t^2,t^3,\cdots,t^d): t 
\in\mathbb{R}\}$. 
Without loss of generality, let us assume that $C$ contains the first vertex 
(i.e., the vertex corresponding to the minimum value of $t$) of $K^d_{2d}$. 
Note that the number of such unordered pairs $\{C, D\}$ is $\frac{1}{2} {2d 
\choose d} = {2d-1 \choose d-1}$. Let us color the vertices in $C$ and $D$ by 
red and blue, respectively, to obtain $d$ partitions created by the red 
vertices. In particular, the first $d-1$ of these partitions are between two 
adjacent red vertices, and the last one is after the last red vertex. It 
implies 
from Lemma 
\ref{extension} and Lemma \ref{lem:dpach} that the pair of $(d-1)$-simplices 
formed by the vertices in $C$ and $D$ cross if there exists a sequence of 
$d+2$ vertices with alternating colors. Similarly, we obtain from Lemma 
\ref{lem:dpach} and Lemma \ref{lem:dpach1} that the pair of 
$(d-1)$-simplices 
formed by the vertices in $C$ and $D$ do not cross if there does not exist 
any sequence of $d+2$ vertices with alternating colors. 

When $d$ is even, the number of disjoint vertex sets $\{C,D\}$ that do not 
contain any subsequence of length $d+2$ having alternating colors is equal 
to the number of ways $d$ blue vertices can be distributed among $d$ partitions 
such that at most $\frac{d}{2}$ of the partitions are non-empty. This 
number is equal to 
$\sum\limits_{i=1}^{\frac{d}{2}}\dbinom{d}{i}\dbinom{d-1}{i-1}$. When 
$d$ is odd, the number of disjoint vertex sets $\{C,D\}$ 
that do not 
contain any subsequence of length $d+2$ having alternating colors is equal 
to the number of ways $d$ blue vertices can be distributed among $d$ partitions 
such that at most $\floor{\frac{d}{2}}$ of the first $d-1$ partitions are 
non-empty. This number is equal to 
$\sum\limits_{i=1}^{\floor{\frac{d}{2}}}\dbinom{d-1}{i}\Bigg(\dbinom{
d-1}{i-1} + \dbinom{d-1}{i}\Bigg)+1$. 
Hence, the total number of crossing pairs of $(d-1)$-simplices spanned by the 
$2d$ vertices placed on the $d$-dimensional moment curve is
\[c_d^m = \begin{cases}
\dbinom{2d-1}{d-1}-\sum\limits_{i=1}^{\frac{d}{2}}\dbinom{d}{i}\dbinom
{d-1 }{i-1} & if~d~is~even. \\
\dbinom{2d-1}{d-1} -1 - 
\sum\limits_{i=1}^{\floor{\frac{d}{2}}}\dbinom{d-1}{i}\dbinom{d}
{i} & if~d~is~odd.\\ 
\end{cases}\]

\section{Convex Crossing Number in Lower Dimensional Space}
\label{conjecture3d}

In this Section, we show that the number of crossing pairs of hyperedges of $K_6^3$ is $3$ when all the 
vertices of $K_6^3$ are in convex as well as general position in $\mathbb{R}^3$. Note that it implies that
$c^*_3 = 3$, where $c^*_3$ denotes the $3$-dimensional convex crossing number of $K_6^3$. However, we are not aware
of the exact values of $c^*_d$ for $d >3$.
\vspace{.2cm}

\noindent\textbf{Proof of Theorem \ref{convex3d}:}
Let $A$ be the set of vertices of $K_6^3$ that are in convex as well as general position in 
$\mathbb{R}^3$. Let $D(A)$ denote the Gale Transform of 
$A$. Since the points in $A$ are in general position, Lemma \ref{genposi} shows 
that the $6$ vectors in $D(A)$ are in general position in $\mathbb{R}^2$. Since 
the 
points in $A$ are also in convex position, Lemma 
\ref{convexity} implies that these vectors can be partitioned by a line $l$ 
passing through the origin in 
two possible ways, i.e., the number of vectors in the opposite open half-spaces 
created by $l$ can be either $4$ and $2$, or $3$ and $3$. Note that the second 
case is also known as a proper linear separation that corresponds to a crossing 
pair of $2$-simplices spanned by the points in $A$. Without loss of generality, 
let us assume that $l$ partitions the vectors in $D(A)$ in such a way that one 
of the open half-spaces created by $l$ contains $4$ vectors and the other 
contains $2$ vectors. We rotate $l$ clockwise using the origin as the axis of 
rotation until one vector changes its side. Since Lemma 
\ref{convexity} shows that $l$ cannot partition the vectors such that there 
exists $1$ vector on one of its side, this new partition obtained by rotating 
$l$ is a proper linear separation. We again rotate $l$ clockwise using the 
origin as the 
axis of rotation until one vector changes its side to obtain a new partition 
having $4$ vectors on one side and $2$ on the other side. We continue rotating 
$l$ clockwise till we reach the first 
partition to obtain three proper linear separations of the vectors in $D(A)$. 
\qed

\section{Concluding Remarks}
\label{conclusions}
In this paper, we have improved the lower bound on the $d$-dimensional 
rectilinear 
crossing number of $K_{2d}^d$. However, there is still a significant gap 
between the best-known asymptotic lower and upper bounds on this number. Moreover, Anshu and 
Shannigrahi \cite{SA} mentioned that the exact values of the $d$-dimensional 
rectilinear crossing number of $K_{2d}^d$ are not known for $d>4$. Similarly, we mentioned in this paper that 
the exact values of the $d$-dimensional convex crossing number of 
$K_{2d}^d$ are not known for $d > 3$. For $d > 3$, it is also an exciting open problem to prove or disprove 
the following conjecture.

\begin{conjecture}
The placement of $n$ vertices on the $d$-dimensional moment curve maximizes the number of crossing pairs of hyperedges in a $d$-dimensional 
convex drawing of $K^d_{n}$. 
\end{conjecture}

\section*{Acknowledgement}
This work has been supported by Ramanujan Fellowship, Department of Science and
Technology, Government of India, grant number SR/S2/RJN-87/2011.

\end{document}